\documentclass{amsart}

\usepackage{amscd,amsxtra,amsthm}
\usepackage[all]{xy}
\usepackage{etex}
\usepackage{pictex}
\usepackage{graphicx}
\usepackage{float}
\usepackage{comment}
\usepackage{multirow}

\newtheorem{theorem}{Theorem}

\theoremstyle{definition}

\theoremstyle{corollary}
\newtheorem{corollary}[theorem]{Corollary}

\theoremstyle{conjecture}
\newtheorem{conjecture}[theorem]{Conjecture}

\theoremstyle{remark}

\theoremstyle{proposition}

\def\R{\mathbb{R}}

\begin{document}
\parskip10pt
\parindent15pt
\baselineskip15pt    

\title[Gallai-Ramsey Theory for Hypergraphs]{Constructive Methods in Gallai-Ramsey Theory for Hypergraphs}

\author[M. Budden]{Mark Budden}\thanks{The authors would like to thank Prof. Stanislaw Radziszowski of the Rochester Institute of Technology for pointing us to a construction by Prof. Brendan McKay \cite{M} that implies the lower bound $R^2(K_{5}^{(3)}-e;3)\geq 14$. }
\address{Department of Mathematics and Computer Science \\
Western Carolina University \\
Cullowhee, NC 28723 USA}
\email{mrbudden@email.wcu.edu}

\author[J. Hiller]{Josh Hiller}
\address{Department of Mathematics and Computer Science \\ 
Adelphi University \\
Garden City, NY 11530-0701}
\email{johiller@adelphi.edu}

\author[A. Penland]{Andrew Penland}
\address{Department of Mathematics and Computer Science \\
Western Carolina University \\
Cullowhee, NC 28723 USA}
\email{adpenland@email.wcu.edu}


\begin{abstract} Much recent progress in hypergraph Ramsey theory has focused on constructions that lead to lower bounds for the corresponding Ramsey numbers.  In this paper, we consider applications of these results to Gallai colorings.  That is, we focus on the Ramsey numbers resulting from only considering $t$-colorings of the hyperedges of complete $r$-uniform hypergraphs in which no rainbow $K_{r+1}^{(r)}$-subhypergraphs exists.  We also provide new constructions which imply improved lower bounds for many $3$ and $4$-uniform Ramsey numbers and $3$ and $4$-uniform Gallai-Ramsey numbers.\end{abstract}

\maketitle

\section{Introduction}

The study of Ramsey numbers for $r$-uniform hypergraphs has seen many advances in recent years.  Although only one classical hypergraph Ramsey number has been evaluated at this time ($R(K_4^{(3)}, K_4^{(3)};3)=13$, \cite{MR}), there has been considerable progress made in the cases of hypergraph paths, cycles, and trees (e.g., \cite{BHR2}, \cite{BP}, \cite{GR}, \cite{Jack}, \cite{Loh}, \cite{LP}, \cite{LPR}, and \cite{OS}) as well as general constructions (e.g., \cite{BBH}, \cite{BHLS}, and \cite{BHR}).  The goal of the present paper is to investigate some of these results within the framework of Gallai colorings, and to provide constructions which yield improved lower bounds for many Gallai-Ramsey hypergraph numbers, especially in the $3$-uniform case.  

To be precise, recall that an $r$-uniform hypergraph $H=(V,E)$ consists of a nonempty set $V$ of vertices and a set $E$ of hyperedges (unordered $r$-tuples of distinct elements in $V$).  We also assume that the elements in $E$ are distinct.   When $r=2$, this definition agrees with that of a simple graph.  The complete $r$-uniform hypergraph of order $n$, in which every $r$-tuple of elements in $V$ form a hyperedge, is denoted by $K_n^{(r)}$.  A $t$-coloring of an $r$-uniform hypergraph $H$ is an assignment of $t$ colors to the hyperedges of $H$.  A $t$-coloring can be identified with a function  $$\mathcal{C}:E(H)\longrightarrow \{1, 2, \dots , t\}.$$  Letting $H_1, H_2, \dots , H_t$ be $r$-uniform hypergraphs, the hypergraph Ramsey number $R(H_1, H_2, \dots , H_t;r)$ is defined to be the least natural number $p$ such that every $t$-coloring of the hyperedges of $K_p^{(r)}$ contains a monochromatic subhypergraph isomorphic to $H_i$ for some color $i$.  Whenever $H=H_1=H_2=\cdots =H_t$, we denote the corresponding Ramsey number by $R^t(H;r)$.

One direction in which the Ramsey theory for graphs was generalized was by restricting the $t$-colorings one considers.  Specifically, we say that a graph $G$ is rainbow if every edge of $G$ is assigned a unique color.  A Gallai $t$-coloring of the complete graph $K_n^{(2)}$ is a $t$-coloring that does not contain any rainbow triangles (i.e., rainbow $K_3^{(2)}$-subgraphs).  Gallai colorings were developed based on the partitions introduced by Gallai \cite{G} in 1967 and were studied within the framework of Ramsey theory in \cite{BB}, \cite{CG}, \cite{FGJM}, \cite{FGP}, \cite{FMO}, and \cite{GS}.  In light of this work, define the Gallai-Ramsey number $gr(H_1, H_2, \dots , H_t;r)$ to be the least natural number $p$ such that every Gallai $t$-coloring (lacking rainbow $K_{r+1}^{(r)}$-subhypergraphs) contains a monochromatic $H_i$ in some color $i$.  As with Ramsey numbers, when $H=H_1=H_2=\cdots = H_t$, we write $gr^t(H;r)$ for the corresponding Gallai-Ramsey number.  Since $K_{r+1}^{(r)}$ contains $r+1$ hyperedges, we find that $$gr(H_1, H_2, \dots , H_t;r)=R(H_1, H_2, \dots , H_t;r)$$ whenever $t<r+1$.  When $t\ge r+1$, we have only the inequality $$gr(H_1, H_2, \dots , H_t;r)\le R(H_1, H_, \dots , H_t;r).$$  While the above references demonstrate that Gallai-Ramsey numbers have been well-developed when $r=2$, to our knowledge, this concept has not been developed when $r>2$.  

In Section \ref{exist}, we review some recent constructions in hypergraph Ramsey theory.  Our goal is to show how these constructions can be applied to Gallai colorings, resulting in lower bounds for hypergraph Gallai-Ramsey numbers.  In Section \ref{new}, we turn our attention to a new $3$ and $4$-uniform constructions that provide lower bounds for certain Gallai-Ramsey numbers.  We conclude with Section \ref{conclude}, listing some new explicit lower bounds implied by our work and offering a conjecture concerning uniformities greater than $4$.


\section{A review of constructions which preserve Gallai Colorings} \label{exist}

In this section, we focus on recent work in hypergraph Ramsey theory which preserves Gallai colorings. In essence we show that many existing theorems from hypergraph Ramsey theory have direct analogues in hypergraph Gallai-Ramsey theory. 

\subsection{Graph Lifting}

In recent work \cite{BHLS}, the authors examined conditions under which $2$-colorings of graphs can be ``lifted'' to $2$-colorings of $3$-uniform hypergraphs.  The lifting described in \cite{BHLS} preserved complements, allowing Ramsey number results for graphs to imply lower bounds for Ramsey numbers of certain $3$-uniform hypergraphs.  Coupling this lifting with the Stepping-Up Lemma credited to Erd\H{o}s and Hajnal (see \cite{GRS}), one can then obtain lower bounds for hypergraph Ramsey numbers of any uniformity.  Let us now outline the framework necessary to precisely describe these results.

Let $\mathcal{G}_2$ denote the set of graphs of order at least $3$ and $\mathcal{G}_3$ denote the set of $3$-uniform hypergraphs of order at least $3$.  Then the lifting $\varphi :\mathcal{G}_2\longrightarrow \mathcal{G}_3$ is defined to map a graph $G$ to the $3$-unform hypergraph $\varphi (G)$ with vertex set $V(\varphi (G))=V(G)$ and hyperedge set consisting of all $3$-tuples $xyz$ such that the subgraph of $G$ induced by $\{x,y,z\}$ contains exactly one edge or exactly $3$ edges.  One of the main consequences of the structure of this lifting is the following theorem.
\begin{theorem}[Theorem 9 in \cite{BHLS}] \label{lifting} For all $s_1, s_2 \ge 3$, $$R(K_{2s_1-1}^{(3)}-e , K_{2s_2-1}^{(3)}-e;3)\ge R(K_{s_1}, K_{s_2};2).$$
\end{theorem}
\noindent Throughout, we use the notation $K_n^{(r)}-e$ to denote a complete $r$-uniform hypergraph of order $n$ with a single hyperedge removed.  From Theorem \ref{lifting}, it follows that $$R(K_{2s_1-1}^{(3)}, K_{2s_2-1}^{(3)};3)\ge R(K_{s_1}^{(2)}, K_{s_2}^{(2)};2),$$ and this was used to prove new lower bounds for certain non-diagonal $3$-uniform hypergraph Ramsey numbers (see the consequences at the end of \cite{BHLS} for specifics).

Unfortunately, the lifting does not immediately extend to $t$-colorings of complete graphs when $t\ge 3$, as the lifting of a rainbow triangle is not well-defined.  We can avoid this issue by lifting only Gallai $t$-colorings.  
In the $3$-uniform setting, Gallai $t$-colorings of $K_p^{(3)}$ are those that lack rainbow $K_4^{(3)}$-subhypergraphs. 
The following theorem implies that Gallai $t$-colorings of graphs lift to Gallai $t$-colorings of $3$-uniform hypergraphs under $\varphi$.

\begin{theorem}[Erd\H{o}s, Simonovits, S\'os \cite{ESS}] Every Gallai coloring of $K_p^{(2)}$ ($p>2$) contains at most $p-1$ colors.\end{theorem}
\noindent In particular, every Gallai coloring of $K_4^{(2)}$ contains at most $3$ colors.  It follows that every $K_4^{(3)}$-subhypergraph of a hypergraph in the image of $\varphi$ is spanned by hyperedges using at most $3$ colors.  So, every hypergraph coloring in the image of $\varphi$ is Gallai.

The following variation of Theorem 9 in \cite{BHLS} follows immediately when we restrict ourselves to Gallai colorings. 

\begin{theorem} \label{lift} For all $s_i\ge 3$, 
$$gr(K_{2s_1-1}^{(3)}-e, K_{2s_2-1}^{(3)}-e, \dots , K_{2s_t-1}^{(3)}-e; 3)\ge gr(K_{s_1}^{(2)}, K_{s_2}^{(2)}, \dots , K_{s_t}^{(2)};2).$$
\end{theorem}
\noindent In particular, it also follows that $$gr(K_{2s_1-1}^{(3)}, K_{2s_2-1}^{(3)}, \dots , K_{2s_t-1}^{(3)}; 3)\ge gr(K_{s_1}^{(2)}, K_{s_2}^{(2)}, \dots K_{s_t}^{(2)};2).$$  

A well-known result of Chung and Graham \cite{CG} is equivalent to the following theorem. 
\begin{theorem}[Chung and Graham \cite{CG}] \label{CandG} For all $t\ge 1$, $$gr^t(K_3^{(2)};2)=\left\{ \begin{array}{ll} 5^{t/2}+1 & \mbox{if $t$ is even} \\ 2\cdot 5^{(t-1)/2}+1 & \mbox{if $t$ is odd.}\end{array}\right.$$
\end{theorem}
\noindent Combining this result with Theorem \ref{lift}, it follows that $$gr^t(K_{5}^{(3)}-e;3)\ge \left\{ \begin{array}{ll} 5^{t/2}+1 & \mbox{if $t$ is even} \\ 2\cdot 5^{(t-1)/2}+1 & \mbox{if $t$ is odd.}\end{array}\right.$$
This result provides a nice extension of Theorem \ref{CandG}, however, it does not provide a tight lower bound.  In Section \ref{conclude}, we will show that $$gr^4(K_5^{(3)}-e;3)\ge gr(K_4^{(3)}, K_4^{(3)}, K_5^{(3)}-e, K_5^{(3)}-e;3)\geq 170,$$ which is stronger than the bound implied by our extension of Theorem \ref{CandG}.

\subsection{Other Constructive Lower Bounds}

In this section, we show how some general constructions can be used to obtain lower bounds for $r$-uniform Gallai-Ramsey numbers.  Recall that the (weak) chromatic number $\chi (H)$ of an $r$-uniform hypergraph $H$ is the minimum number of colors needed to color the vertices of $H$ such that no hyperedge is monochromatic.  When $r=2$, this definition coincides with that of the usual chromatic number.  The chromatic index $s(H)$ is the smallest cardinality of a color class among all proper vertex colorings of $H$ using $\chi (H)$ colors.  

In \cite{BP}, a general lower bound for Ramsey numbers due to Burr \cite{Burr} was extended to the setting of hypergraphs for the purpose of defining an $n$-good $r$-uniform hypergraph.
Using a similar method of proof, we offer the following theorem concerning $t$-colored $r$-uniform Gallai-Ramsey numbers.

\begin{theorem} \label{genburr}
Let $H_1, H_2, \dots , H_t$ be connected $r$-uniform hypergraphs such that $gr(H_1, H_2,\dots , H_{t-1};r) \ge s(H_t)$.  Then $$gr(H_1, H_2, \dots ,  H_t;r)\ge (\chi (H_t)-1)(gr(H_1, H_2, \dots , H_{t-1};r)-1)+s(H_t).$$
\end{theorem}

\begin{proof}
Let $p=gr(H_1, H_2, \dots , H_{t-1};r)$ and fix a Gallai $(t-1)$-coloring of $K_{p-1}^{(r)}$ that lacks a monochromatic $H_i$ in color $i$, for all $i\in \{1, 2, \dots , t-1\}$.  Consider $\chi(H_{t})-1$ copies of this $K_{p-1}^{(r)}$, along with a single Gallai $(t-1)$-colored $K_{s(H_{t})-1}^{(r)}$, formed by removing $p-s(H_t)$ vertices arbitrarily from the $K_{p-1}^{(r)}$. The result is a Gallai $(t-1)$-colored $$(\chi(H_{t})-1)K_{p-1}^{(r)}\cup K_{s(H_{t})-1}^{(r)},$$ whose vertices we then interconnect with color $t$.   It is clear that no rainbow $K_{r+1}^{(r)}$-subhypergraphs exist entirely within any of the Gallai $(t-1)$-colored complete hypergraphs.  Any $K_{r+1}^{(r)}$-subhypergraphs that include vertices from two distinct complete subhypergraphs necessary contain at least two hyperedges of the same color.  Hence the resulting $t$-coloring of $K_{(\chi (H_t)-1)(p-1)+s(H_t)-1}^{(r)}$ is a Gallai coloring.  By construction, this coloring does not contain a monochromatic copy of $H_i$ in  any color $i\in \{1, 2, \dots ,t-1\}$. It remains to be argued that we have not produced a copy of $H_t$ in color $t$.  In the case where $s(H_t)=1$, a proper vertex coloring of any subhypergraph spanned by hyperedges in color $t$ can be achieved by coloring the vertices according to which copy of $K_{p-1}^{(r)}$ they are in.  So, any such subhypergraph is $(\chi (H_t)-1)$-properly colorable and $H_t$ is not such a subhypergraph.  In the case where $s(H_t)>1$, coloring the vertices according to which complete subhypergraph they are in produces a proper vertex coloring of every subhypergraph spanned by hyperedges using color $t$.  Since such a proper coloring using $\chi (H_t)$ colors contains a color class with cardinality $s(H_t)-1$, we find that $H_t$ is not such a subhypergraph.  In either case, we find that $$gr(H_1, H_2,  \dots , H_t;r)>(\chi(H_t)-1)(p-1)+s(H_t)-1,$$ completing the proof of the theorem. 
\end{proof}

While the lower bounds implied by Theorem \ref{genburr} are not exceptionally strong when considering only complete hypergraphs, they may be useful when considering minimally connected hypergraphs (e.g., see \cite{BHP}).  This theorem may also be useful in extending the notion of $n$-good hypergraphs (see \cite{BP}) to Gallai colorings.

Another constructive bound for hypergraph Ramsey numbers was given in \cite{BBH}.  In this case, properties of the lexicographic product of hypergraphs were exploited.  If $H_1$ and $H_2$ are $r$-uniform hypergraphs with $r>2$, then the lexicographic product $H_1[H_2]$ is the hypergraph defined to have vertex set $V(H_1)\times V(H_2)$ and hyperedge set $E(H_1[H_2])$ given by $$\left\{ (a_1, b_1)(a_2, b_2)\cdots (a_r, b_r) \ \left| \ a_1a_2\cdots a_r\in E(H_1) \ \mbox{or} \ \left( \begin{array}{c} a_1=a_2=\cdots a_r \ \mbox{and} \\ b_1 b_2\cdots b_r\in E(H_2)\end{array}\right)\right.\right\}.$$
Of course, the lexicographic product of hypergraphs is not commutative.  The usefulness of this product follows from the following properties of the clique number, denoted by $\omega$, which were proved in Theorems 3 and 4 of \cite{BBH}:
\begin{equation} \label{clique1} \omega (H_1[H_2])=\max(\omega (H_1), \omega (H_2))\end{equation}
and
\begin{equation} \label{clique2} \omega (\overline{H_1[H_2]})=\omega (\overline{H_1})\omega (\overline{H_2}).\end{equation}

\begin{theorem}\label{lex}
If $r\ge 3$ and $p_i,q_i\ge r-1$ for all $1\le i\le t$, then $$gr(K_{\max(p_1, q_1)+1}^{(r)}, K_{\max(p_2, q_2)+1}^{(r)}, \dots , K_{\max(p_{t-1}, q_{t-1})+1}^{(r)}, K_{p_tq_t+1}^{(r)};r)\ge \qquad \qquad \qquad$$ $$\qquad \quad(gr(K_{p_1+1}^{(r)}, K_{p_2+1}^{(r)}, \dots , K_{p_t+1}^{(r)};r)-1)(gr(K_{q_1+1}^{(r)}, K_{q_2+1}^{(r)}, \dots , K_{q_t+1}^{(r)};r)-1)+1.$$
\end{theorem}
\begin{proof}
The proof  of Theorem \ref{lex} is structured the same as the analogous result in \cite{BBH}, but it is also necessary to justify why the construction does not contain  any rainbow $K_{r+1}^{(r)}$-subhypergraphs.  Let $$m:=gr(K_{p_1+1}^{(r)}, K_{p_2+1}^{(r)}, \dots , K_{p_t+1}^{(r)};r) \quad \mbox{and} \quad n:=gr(K_{q_1+1}^{(r)}, K_{q_2+1}^{(r)}, \dots , K_{q_t+1}^{(r)};r).$$  Then there exists a Gallai $t$-coloring $\mathcal{C}_1$ of $K_{m-1}^{(r)}$ that lacks a monochromatic clique of order $p_i+1$ in color $i$ for every $i\in \{1, 2, \dots , t\}$ and there exists a Gallai $t$-coloring $\mathcal{C}_2$ of $K_{n-1}^{(r)}$ that lacks a monochromatic clique of order $q_i+1$ in color $i$ for every $i\in \{ 1, 2, \dots , t\}$.  We now construct a Gallai $t$-coloring $\mathcal{C}$ of $K_{(m-1)(n-1)}^{(r)}$ by first identifying the vertices with those in $V(K_{m-1}^{(r)})\times V(K_{n-1}^{(r)})$.  Then assign the hyperedge $$(a_1, b_1)(a_2, b_2)\cdots (a_r, b_r)$$ the color $j\in \{ 1, 2, \dots , t\}$ if either $a_1a_2\cdots a_r\in E(K_{m-1}^{(r)})$ has color $j$ in $\mathcal{C}_1$ or if $a_1=a_2=\cdots a_r$ and $b_1b_2\cdots b_r\in E(K_{n-1}^{(r)})$ has color $j$ in $\mathcal{C}_2$.  All remaining hyperedges are assigned color $t$.  For each $j$, define $H_j$ and $H_j'$ to be the subhypergraphs of $K_{m-1}^{(r)}$ and $K_{n-1}^{(r)}$ spanned by the hyperedges of color $j$, respectively.  Whenever $1\le j\le t-1$, the subhypergraph of $K_{(m-1)(n-1)}^{(r)}$ spanned by the hyperedges of color $j$ is isomorphic to $H_j[H_j']$.  Such a subhypergraph does not contain any cliques of order $\max(p_j, q_j)+1$ in color $j\in S:=\{ 1, 2, \dots , t-1\}$ by (\ref{clique1}).  The subhypergraph spanned by color $t$ in $K_{(m-1)(n-1)}^{(r)}$ is isomorphic to the complement of $$\mathop{\bigcup}\limits_{j\in S} H_j\bigg[\mathop{\bigcup}\limits_{j\in S} H_j'\bigg],$$ and hence, lacks a clique of order $p_tq_t+1$ in color $t$ by (\ref{clique2}).  The theorem will now follow from showing that no rainbow $K_{r+1}^{(r)}$ exists in our construction.  Consider a set $$W=\{(a_1, b_1), (a_2, b_2), \dots , (a_{r+1}, b_{r+1})\}$$ of distinct vertices.  If these vertices all have distinct $a_i$, then the subhypergraph induced by $W$ is colored according to $\mathcal{C}_1$, and hence, is Gallai colored.  If at least two $a_i$ (but not all) are equal, then at least two of the hyperedges in the subhypergraph induced by $W$ receive color $t$.  Finally, if all $a_i$ are equal, then the subhypergraph induced by $W$ is colored according to $\mathcal{C}_2$.  From these cases, we find that we have produced a Gallai $t$-coloring of $K_{(m-1)(n-1)}^{(r)}$.
\end{proof}

Using the observation that $$gr(K_{n_1}^{(r)}, K_{n_2}^{(r)}, \dots , K_{n_i}^{(r)};r)=gr(K_{n_1}^{(r)}, K_{n_2}^{(r)}, \dots , K_{n_i}^{(r)}, \underbrace{K_r^{(r)}, K_{r}^{(r)}, \dots, K_r^{(r)}}_{t-i \ terms};r)$$ and $$gr(K_{n_{i+1}}^{(r)}, K_{n_{i+2}}^{(r)},\dots, K_{n_t}^{(r)};r)=gr(\underbrace{K_r^{(r)}, K_r^{(r)}, \dots , K_r^{(r)}}_{i\ terms}, K_{n_{i+1}}^{(r)}, K_{n_{i+2}}^{(r)},\dots, K_{n_t}^{(r)};r),$$ we obtain the following corollary, which gives a Gallai-coloring version of a theorem of Xu, Xie, Exoo, and Radziszowski \cite{XXER}.  

\begin{corollary} \label{ExooRad}
If $r\ge 3$ and each $n_j \ge r$, then \begin{align}gr(K_{n_1}^{(r)}, K_{n_2}^{(r)}, \dots , K_{n_{t-1}}^{(r)}, K_{(n_t-1)(r-1)+1}^{(r)};r)&\notag\\ \ge (gr(K_{n_1}^{(r)}, K_{n_2}^{(r)}, \dots , K_{n_i}^{(r)};r)-1)&(gr(K_{n_{i+1}}^{(r)}, K_{n_{i+2}}^{(r)},\dots, K_{n_t}^{(r)};r)-1)+1.\notag \end{align}
\end{corollary}

\section{New Constructions for $3$-Uniform and $4$-Uniform Hypergraph Gallai-Ramsey Numbers} \label{new}

In this section, we present new constructions which yield improved lower bounds for several multicolored Gallai-Ramsey numbers. 


\begin{theorem}\label{3construct}  Let $H_1, H_2, \dots , H_t$ be $3$-uniform hypergraphs with orders greater than $3$ such that each $H_i$ is isomorphic to either a complete hypergraph or a complete hypergraph with a single hyperedge removed.  If $t\ge 1$, then  
  $$ gr(H_1, H_2, \dots , H_t, K_4^{(3)}, K_4^{(3)};3)\ge (gr(H_1, H_2, \dots , H_t;3)-1)^2+1.$$
\end{theorem}

\begin{proof}
 Let $m=gr(H_1, H_2, \dots , H_t;3)$, then there exists a Gallai $t$-coloring of a $K_{m-1}^{(3)}$ that avoids a monochromatic $H_i$ in color $i$ for all $i\in\{1,2,...,t\}$.   We union $m-1$ disjoint copies of this $K_{m-1}^{(3)}$ and label them $A_1,A_2,...,A_m$.   Within each $A_i$, label the vertices $$(1,i),(2,i),...,(m-1,i)$$ so that the color of the hyperedge $(j,i)(k,i)(\ell ,i)$ is independent of the choice of $i$.  For a hyperedge $e=abc$, where not all vertices come from the same same copy of $A_i$, color $e$ according to the following rules.
\begin{enumerate}
    \item Suppose that no two vertices in $e=abc$ come from the same $A_i$.  Let $a\in V(A_j)$, $b\in V(A_k)$, and $c\in V(A_\ell)$.  Then color $e$ according to the color of the hyperedge $(j, i)(k,i)(\ell, i)$ in $A_i$.  
    \item If $a,b\in V(A_j)$ and $c\in V(A_k),$ and $j<k$, then assign color $t+1$ to $e$.
    \item If $a,b\in V(A_j)$ and $c\in V(A_k),$ and $j>k$, then assign color $t+2$ to $e$.
\end{enumerate}
We leave it as an exercise to confirm that these rules produce a Gallai $(t+2)$-coloring of $K_{(m-1)^2}^{(3)}$.  We can also confirm that no monochromatic $K_4^{(3)}$ is spanned by hyperedges in colors $t+1$ and $t+2$ as any such subhypergraph contains exactly two vertices in some $A_i$ and two vertices in some $A_j$, with $i\ne j$.  The subhypergraph then contains exactly two hyperedges in color $t+1$ and two hyperedges in color $t+2$.
It remains to be shown that this hypergraph coloring lacks a monochromatic copy of $H_i$ in color $i$, for all $i\in \{1,2,\dots , t\}$.  Let $S$ be a subset of vertices with cardinality $n\ge 4$.  We must consider several cases.  \\
\noindent \underline{Case 1:} If all vertices in $S$ come from distinct $A_i$, then the hyperedges in the subhypergraph induced by $S$ are colored according to rule $(1)$ above, preventing a monochromatic copy of $H_i$ from being produced in color $i$ for all $i\in\{ 1,2,\dots , t\}$. \\
\noindent \underline{Case 2:} If $S\subseteq V(A_i)$ for some $i$, then the subhypergraph induced by $S$ does not contain a monochromatic $H_i$ in color $i$, for any $i\in \{1,2,\dots , t\}$, by the assumption made when constructing $A_i$.\\
\noindent \underline{Case 3:} Suppose that $S\subseteq V(A_i)\cup V(A_j)$ with $i\ne j$ and $S\cap V(A_i)\ne \emptyset$ and $S\cap V(A_j)\ne \emptyset$.   We have assumed that $S$ has cardinality at least $4$.  In the case where it is equal to $4$ and  $|S\cap V(A_i)|=2=|S\cap V(A_j)|$, we find that the subhypergraph induced by $S$ contains two hyperedges in color $t+1$ and two hyperedges in color $t+2$.  Otherwise, one of $S\cap V(A_i)$ and $S\cap V(A_j)$ has order at least $3$ and the other has order at least $1$.  The subhypergraph induced by $S$ then contains at least $3$ hyperedges in colors $t+1$ and $t+2$.  In both of these subcases, no monochromatic copy of $H_i$ is produced in color $i$, for any $i\in \{1, 2, \dots , t\}$. \\
\noindent \underline{Case 4:} Suppose that $S$ contains vertices from at least $3$ distinct $A_i$, one of which contains at least two vertices from $S$ (so that we are not in Case 1).  Without loss of generality, suppose that $a,b,c,d \in S$ satisfy $a,b\in A_i$ and $c,d\not\in A_i$ for some $i$.  The hyperedges $abc$ and $abd$ receive colors $t+1$ or $t+2$, preventing the subhypergraph induced by $S$ from containing a monochromatic copy of $H_i$ in color $i$, for all $i\in \{1,2,\dots , t\}$.  \\
\noindent This completes the proof.
\end{proof}

\begin{theorem}\label{4construct}
Let $H_1, H_2, \dots , H_t$ be $4$-uniform hypergraphs with orders greater than $4$ such that each $H_i$ is isomorphic to either a complete hypergraph or a complete hypergraph with a single hyperedge removed.  If $t\ge 1$, then  
  $$ gr(H_1, H_2, \dots , H_t, K_5^{(4)}, K_5^{(4)};4)\ge (gr(H_1, H_2, \dots , H_t;4)-1)^2+1.$$
\end{theorem}

\begin{proof}
 Let $m=gr(H_1, H_2, \dots , H_t;4)$, then there exists a Gallai $t$-coloring of a $K_{m-1}^{(4)}$ that avoids a monochromatic $H_i$ in color $i$ for all $i\in\{1,2,...,t\}$.   We union $m-1$ disjoint copies of this $K_{m-1}^{(4)}$ and label them $A_1,A_2,...,A_m$.   Within each $A_i$, label the vertices $$(1,i),(2,i),...,(m-1,i)$$ so that the color of the hyperedge $(g,i)(j,i)(k,i)(\ell ,i)$ is independent of the choice of $i$.  For a hyperedge $e=abcd$, where not all vertices come from the same same copy of $A_i$, color $e$ according to the following rules.
\begin{enumerate}
    \item Suppose that no two vertices in $e=abcd$ come from the same $A_i$.  Let $a\in V(A_h)$, $b\in V(A_j)$, $c\in V(A_k)$, and $d\in V(A_\ell)$.  Then color $e$ according to the color of the hyperedge $(h,i)(j, i)(k,i)(\ell, i)$ in $A_i$.  
    \item If $a,b\in V(A_j)$ and $c,d\in V(A_k)$ for $j\ne k$, then assign color $1$ to $e$.
    \item If $a,b, c\in V(A_j)$ and $d\in V(A_k),$ for $j\ne k$, then assign color $t+1$ to $e$.
    \item If $a,b\in V(A_j)$, $c\in V(A_k)$, and $d\in V(A_\ell)$ for $j\ne k,\ell$ and $k\ne \ell$, then assign color $t+2$ to $e$.
\end{enumerate}
It is easy to confirm that these rules produce a Gallai $(t+2)$-coloring of $K_{(m-1)^2}^{(4)}$.  We can also confirm that no monochromatic $K_5^{(4)}$ is spanned by hyperedges in colors $t+1$ or $t+2$. 
It remains to be shown that this hypergraph coloring lacks a monochromatic copy of $H_i$ in color $i$, for all $i\in \{1,2,\dots , t\}$.  Let $S$ be a subset of vertices with cardinality $n\ge 5$.  We consider several cases.  \\
\noindent \underline{Case 1:} If all vertices in $S$ come from distinct $A_i$, then the hyperedges in the subhypergraph induced by $S$ are colored according to rule $(1)$ above, preventing a monochromatic copy of $H_i$ from being produced in color $i$ for all $i\in\{ 1,2,\dots , t\}$. \\
\noindent \underline{Case 2:} If $S\subseteq V(A_i)$ for some $i$, then the subhypergraph induced by $S$ does not contain a monochromatic $H_i$ in color $i$, for any $i\in \{1,2,\dots , t\}$, by the assumption made when constructing $A_i$.\\
\noindent \underline{Case 3:} Suppose that $S\subseteq V(A_i)\cup V(A_j)$ with $i\ne j$ and $S\cap V(A_i)\ne \emptyset$ and $S\cap V(A_j)\ne \emptyset$.   We have assumed that $S$ has cardinality at least $5$.  In the case where it is equal to $5$ and  $|S\cap V(A_i)|=2$ and $|S\cap V(A_j)|=3$, we find that the subhypergraph induced by $S$ contains three hyperedges in color $t+1$ and two hyperedges in color $1$, preventing a monochromatic $HJ_i$ in color $i$ for all $i\in \{ 1, 2, \dots , t\}$.  In the case where it is equal to $5$ and $|S\cap V(A_i)|=1$ and $|S\cap V(A_j)|=4$, four hyperedges receive color $t+1$ and one hyperedge has a color from $\{1,2,\dots , t\}$.  Again, no monochromatic copy of $H_i$ exists in color $i$ for any $i\in \{1,2,\dots, t\}$. \\
\noindent \underline{Case 4:} Suppose that $S$ contains vertices from at least $3$ distinct $A_i$, one of which contains at least two vertices from $S$ (so that we are not in Case 1).  In the case where $a,b,c,d,e \in S$ satisfy $a,b\in A_i$, $c,d\in A_j$, and $e\in A_k$, where no two of $i,j,k$ are equal, one hyperedge has color $1$ and four hyperedges have color $t+2$. There is no $K_5^{(4)}$ in color $t+2$ or a monochromatic $H_i$ in color $i$ for  any $i\in \{1,2,\dots ,t\}$.    In the case where $a,b,c,d,e \in S$ satisfy $a,b,c\in A_i$, $d\in A_j$, and $e\in A_k$, there are three hyperedges in color $t+2$ and two hyperedges in color $t+1$, again, preventing the appropriate subhypergraphs. \\
\noindent \underline{Case 5:} Suppose that $S$ contains vertices from at least $4$ distinct $A_i$, but not five.  Without loss of generality, assume that $a,b\in V(A_i)$, $c\in V(A_j)$, $d\in V(A_k)$, and $e\in V(A_\ell)$.  Then there are two hyperedges in colors $\{1,2,\dots , t\}$ and three hyperedges in color $t+2$, preventing the appropriate subhypergraphs.  \\
\noindent This completes the proof.
\end{proof}



\section{Explicit Lower Bounds and Future Directions} \label{conclude}

We conclude by explicitly stating the best known $3$ and $4$-uniform lower bounds implied by the constructions contained in this paper.  While we have focused on Gallai-Ramsey numbers, it should be pointed out that Theorems \ref{genburr}, \ref{3construct}, and \ref{4construct} still hold when the Gallai-Ramsey number is replaced by the corresponding Ramsey number in the theorems' statements.  The proofs are the same, but do not require the extra step of showing that Gallai colorings are preserved.  This is particularly useful when we apply the Ramsey number-version of Theorem \ref{3construct} to the known lower bound $R^4(K_5^{(3)};3)\ge 131,073$ (see \cite{BHR}) to obtain $$R(K_4^{(3)}, K_4^{(3)}, K_5 ^{(3)}, K_5 ^{(3)}, K_5 ^{(3)}, K_5 ^{(3)}; 3)\ge 17,179,869,185.$$

Since $2$ and $3$-color Gallai-Ramsey numbers correspond with Ramsey numbers in the hypergraph setting, we can build off of known lower bounds for the corresponding Ramsey numbers, which can be found in Section 7.1 of Radziszowski's dynamic survey \cite{Rad}.  In Figure 3 of \cite{M}, two examples are provided of $2$-colorings of $K_{13}^{(3)}$ that lack a monochromatic $K_5^{(3)}-e$.  It follows that $$R^2(K_5^{(3)}-e;3)\ge 14.$$  Applying Theorem \ref{3construct}, we find that $$gr(K_4^{(3)}, K_4^{(3)}, K_5^{(3)}-e, K_5^{(3)}-e;3)\ge 170,$$ which surpasses the bounds implied by Theorems \ref{lift} and \ref{CandG}. 

From \cite{Rad}, the following lower bounds for $3$-uniform hypergraph Ramsey numbers are known.
\begin{align} R^2 (K_4^{(3)}-e;3)&=7 \quad \mbox{\ \cite{Rad}} \label{4-e,4-e} \\ R(K_4^{(3)}-e, K_4^{(3)};3)&=8 \quad \mbox{\ \cite{Sob}} \label{4-e,4}  \\R^3(K_4^{(3)}-e;3)&\ge 13 \quad \mbox{\cite{Ex1}} \label{3K4-e}\\ R^2(K_4^{(3)};3)&=13  \quad \mbox{\cite{MR}} \label{4,4} \\ R(K_4^{(3)}, K_5^{(3)};3)&\ge 35 \quad \mbox{\cite{Dyb2}} \label{4,5} \\ R(K_4^{(3)}, K_6^{(3)};3)&\ge 58 \quad \mbox{\cite{Ex8}} \label{4,6} \\ R^2(K_5^{(3)};3)&\ge 82 \quad \mbox{\cite{Ex8}} \label{5,5} \\ R^3(K_4^{(3)};3)&\ge 56 \quad \mbox{\cite{Ex8}} \label{4,4,4} \\ R^3(K_5^{(3)};3)&\ge 163 \ \  \mbox{\cite{BHR}} \label{5,5,5}  \end{align} 
Our results then imply the lower bounds given in Figure \ref{3uni}.
\begin{figure}
\begin{align} \mbox{Theorem \ref{3construct} and (\ref{4-e,4-e})} \quad &\Longrightarrow \quad gr(K_4^{(3)}-e, K_4^{(3)}-e, K_4^{(3)}, K_4^{(3)};3)\ge 37 \notag \\
\mbox{Theorem \ref{3construct} and (\ref{4-e,4})} \quad &\Longrightarrow \quad gr(K_4^{(3)}-e, K_4^{(3)}, K_4^{(3)}, K_4^{(3)};3)\ge 50 \notag \\
\mbox{Theorem \ref{3construct} and (\ref{3K4-e})} \quad &\Longrightarrow \quad gr(K_4^{(3)}-e, K_4^{(3)}-e, K_4^{(3)}-e, K_4^{(3)}, K_4^{(3)};3)\ge 145 \notag \\
\mbox{Theorem \ref{3construct} and (\ref{4,4})} \quad &\Longrightarrow \quad gr^4(K_4^{(3)};3)\ge 145 \notag \\
\mbox{Theorem \ref{3construct} and (\ref{4,5})} \quad &\Longrightarrow \quad gr(K_4^{(3)}, K_4^{(3)}, K_4^{(3)}, K_5^{(3)};3)\ge 1,157 \notag \\
\mbox{Corollary \ref{ExooRad} and (\ref{4,5})} \quad &\Longrightarrow \quad gr(K_4^{(3)}, K_5^{(3)}, K_5^{(3)}, K_{7}^{(3)};3)\ge 1,157 \notag \\
\mbox{Corollary \ref{ExooRad}, (\ref{4,5}), and (\ref{5,5})} \quad &\Longrightarrow \quad gr(K_4^{(3)}, K_5^{(3)}, K_5^{(3)}, K_{9}^{(3)};3)\ge 2,755 \notag \\
\mbox{Corollary \ref{ExooRad}, (\ref{4,5}), and (\ref{5,5})} \quad &\Longrightarrow \quad gr(K_5^{(3)}, K_5^{(3)}, K_5^{(3)}, K_{7}^{(3)};3)\ge 2,755 \notag
\\ \mbox{Theorem \ref{3construct} and (\ref{4,4,4})} \quad &\Longrightarrow \quad gr^5(K_4^{(3)};3)\ge 3,026 \notag 
\\ \mbox{Corollary \ref{ExooRad} and (\ref{4,6})} \quad &\Longrightarrow \quad gr(K_4^{(3)}, K_6^{(3)}, K_6^{(3)}, K_{7}^{(3)};3)\ge 3,250 \notag 
\\ \mbox{Corollary \ref{ExooRad} and (\ref{4,6})} \quad &\Longrightarrow \quad gr(K_4^{(3)}, K_4^{(3)}, K_6^{(3)}, K_{11}^{(3)};3)\ge 3,250 \notag 
\\ \mbox{Corollary \ref{ExooRad}, (\ref{4,6}), and (\ref{5,5})} \quad &\Longrightarrow \quad gr(K_4^{(3)}, K_5^{(3)}, K_5^{(3)}, K_{11}^{(3)};3)\ge 4,618 \notag
\\ \mbox{Corollary \ref{ExooRad}, (\ref{4,6}), and (\ref{5,5})} \quad &\Longrightarrow \quad gr(K_4^{(3)}, K_5^{(3)}, K_6^{(3)}, K_{7}^{(3)};3)\ge 4,618 \notag
\\ \mbox{Corollary \ref{ExooRad}, (\ref{4,6}), and (\ref{5,5})} \quad &\Longrightarrow \quad gr(K_5^{(3)}, K_5^{(3)}, K_6^{(3)}, K_{7}^{(3)};3)\ge 4,618 \notag
\\ \mbox{Corollary \ref{ExooRad} and (\ref{5,5})} \quad &\Longrightarrow \quad gr(K_5^{(3)}, K_5^{(3)}, K_5^{(3)}, K_{9}^{(3)};3)\ge 6,565 \notag 
\\ \mbox{Theorem \ref{3construct} and (\ref{5,5,5})} \quad &\Longrightarrow \quad gr(K_5^{(3)},K_5^{(3)},K_5^{(3)},K_4^{(3)},K_4^{(3)};3)\ge 26,245 \notag
\end{align} \caption{Lower bounds for some $3$-uniform hypergraph Gallai-Ramsey numbers.}\label{3uni}\end{figure}
Of course, one can obtain other $t$-color Gallai-Ramsey number lower bounds by successively applying the above results with Corollary \ref{ExooRad} and Theorem \ref{3construct}.  

In the $4$-uniform case, it is known that \begin{equation} R^2(K_5^{(4)};4)\ge 34 \quad \mbox{\cite{Ex8}}.\label{5,5;4} \end{equation}  Combining this bound with Theorem \ref{4construct} implies the following:  $$gr^4(K_5^{(4)};4)\ge 1,090.$$  In \cite{SonYL}, it was shown that $$R(K_{p}^{(4)}, K_{q}^{(4)};4)\ge 2(R(K_{p-1}^{(4)}, K_q^{(4)};4))-1$$ for $p,q\ge 5$.  Applying this inequality to (\ref{5,5;4}), it follows that $$R(K_5^{(4)}, K_6^{(4)};4)\ge 67,$$ which combined with Theorem \ref{4construct} gives $$gr(K_5^{(4)}, K_5^{(4)}, K_5^{(4)}, K_6^{(4)};4)\ge 4,357.$$

Lastly, Theorems \ref{3construct} and \ref{4construct} taken together provide evidence in support of the following conjecture.  The difficulty in proving this conjecture follows from the number of cases that arise as the uniformity increases.

\begin{conjecture}\label{rconstruct}
Let $H_1, H_2, \dots , H_t$ be $r$-uniform hypergraphs, $r>2$,  with orders greater than $r$ such that each $H_i$ is isomorphic to either a complete hypergraph or a complete hypergraph with a single hyperedge removed.  If $t\ge 1$, then  
  $$ gr(H_1, H_2, \dots , H_t, K_{r+1}^{(r)}, K_{r+1}^{(r)};r)\ge (gr(H_1, H_2, \dots , H_t;r)-1)^2+1.$$

\end{conjecture}


\bibliographystyle{amsplain}

\end{document}